\documentclass[draft]{amsart}
\usepackage{amsmath}
\usepackage{amssymb}
\usepackage{graphicx}
\newtheorem{theorem}{Theorem}[section]
\newtheorem{lemma}[theorem]{Lemma}

\theoremstyle{definition}

\newtheorem{remark}[theorem]{Remark}

\numberwithin{equation}{section}

\begin{document}

\title[On the optimal Voronoi partitions for Ahlfors-David measures]{On the optimal Voronoi partitions for Ahlfors-David measures with respect to the geometric mean error}
\author{Sanguo Zhu, Youming Zhou}

\address{School of Mathematics and Physics, Jiangsu University
of Technology\\ Changzhou 213001, China.}
\email{sgzhu@jsut.edu.cn}


\begin{abstract}
Let $\mu$ be an Ahlfors-David probability measure on $\mathbb{R}^q$ with support $K_\mu$. For every $n\geq 1$, let $C_n(\mu)$ denote the collection of all the $n$-optimal sets for $\mu$ with respect to the geometric mean error. We prove that, there exist constants $d_1,d_2>0$, such that for each $n\geq 1$, every $\alpha_n\in C_n(\mu)$ and an arbitrary Voronoi partition $\{P_a(\alpha_n)\}_{a\in\alpha_n}$ with respect to $\alpha_n$, we have
\[
d_1n^{-1}\leq\min_{a\in\alpha_n}\mu(P_a(\alpha_n))\leq\max_{a\in\alpha_n}\mu(P_a(\alpha_n))\leq d_2n^{-1}.
\]
Moreover, we prove that for each $a\in\alpha_n$, the set $P_a(\alpha_n)$ contains a closed ball of radius $d_3|P_a(\alpha_n)\cap K_\mu|$ which is centered at $a$, where $d_3$ is a constant and $|B|$ denotes the diameter of a set $B\subset\mathbb{R}^q$. Some estimates for the measure and the geometrical size of the elements of a Voronoi partition with respect to an $n$-optimal set are established in a more general context. These estimates also allow us to drop an additional condition in our previous work on self-similar measures.
\end{abstract}

\maketitle

\section{Introduction}

Let $\nu$ be a Borel probability measure on $\mathbb{R}^q$. The quantization problem for $\nu$ is concerned with the approximation of $\nu$ by discrete measures of finite support in $L_r$-metrics. This problem has a deep background in information theory and engineering technology such as signal processing and pattern recognition \cite{BW:82,GN:98}. We refer to \cite{GL:00,GL:04} for rigorous mathematical theory of the quantization problem. In the past decades, this problem has attracted great interest of mathematicians (cf. \cite{GL:00,GL:04,GL:08,GL:12,GPM:02,KZ:15,KZ:16,Kr:13,PG:97,PK:01,PK:04}).
\subsection{Some definitions and basic facts}
Let $r\in [0,\infty)$ and $k\in\mathbb{N}$. Let $d$ denote the Euclidean metric on $\mathbb{R}^q$. For every $k\geq 1$, let $\mathcal{D}_k:=\{\alpha\subset\mathbb{R}^q:1\leq {\rm card}(\alpha)\leq k\}$. For $x\in\mathbb{R}^q$ and $\alpha\subset\mathbb{R}^q$, let $d(x,\alpha):=\inf_{a\in\alpha}d(x,a)$. The $k$th quantization error for $\nu$ of order $r$ can be defined by
\begin{eqnarray}\label{quanerror}
e_{k,r}(\nu)=\left\{\begin{array}{ll}\big(\inf\limits_{\alpha\in\mathcal{D}_k}\int d(x,\alpha)^{r}d\nu(x)\big)^{1/r}&r>0\\
\inf\limits_{\alpha\in\mathcal{D}_k}\exp\big(\int\log d(x,\alpha)d\nu(x)\big)&r=0\end{array}\right..
\end{eqnarray}
One may see \cite{GL:00} for some equivalent definitions and interpretations in various contexts. For $r\in[1,\infty)$, $e_{n,r}(\nu)$ is equal to the minimum error when approximating $\nu$ by discrete probability measures supported on at most $n$ points in the $L_r$-metrics.

A set $\alpha\in\mathcal{D}_k$ is called a $k$-optimal set for $\nu$ of order $r$, if the infimum in (\ref{quanerror}) is attained at $\alpha$. We call the points in such an $\alpha$ $k$-\emph{optimal points}. As in \cite{GL:00,GL:04}, we denote the collection of all the $k$-optimal sets for $\nu$ of order $r$ by $C_{k,r}(\nu)$ and simply write $C_k(\nu)$ for $C_{k,0}(\nu)$. For $r>0$, $C_{k,r}(\nu)\neq\emptyset$ if $\int |x|^rd\nu(x)<\infty$; and $C_k(\nu)$ is non-empty if the following condition is satisfied:
\[
\int_0^1 s^{-1}\sup_{x\in\mathbb{R}^q}\nu(B(x,s))ds<\infty.
\]
In particular, $C_k(\nu)\neq\emptyset$ if for some constants $C,t>0$, we have
\[
\sup_{x\in\mathbb{R}^q}\nu(B(x,\epsilon))\leq C\epsilon^t
 \]
for every $\epsilon>0$. Furthermore, whenever the support $K_\nu$ of $\nu$ is an infinite set, we have that $e_{n,r}(\nu)$ is strictly decreasing as $n$ increases and ${\rm card}(\alpha_n)=n$ for every $\alpha_n\in C_{n,r}(\nu)$. One can see Theorem 4.12 of \cite{GL:00} and Theorem 2.4 of \cite{GL:04} for more detailed information.

Let $\alpha$ be a non-empty finite subset of $\mathbb{R}^q$.  For each $a\in\alpha$, the Voronoi region  generated by $a$ with respect to $\alpha$ is defined by
\begin{equation}\label{vregion}
W(a|\alpha):=\{x\in\mathbb{R}^q:d(x,\alpha)=d(x,a)\}.
\end{equation}
A Voronoi partition (VP) with respect to  the set $\alpha$ is defined to be a Borel partition $\{P_a(\alpha)\}_{a\in\alpha}$ of $\mathbb{R}^q$ such that
\[
P_a(\alpha)\subset W(a|\alpha),\;a\in\alpha.
\]
Let us call a VP with respect to an $\alpha\in C_{k,r}(\nu)$ a $k$-\emph{optimal Voronoi partition}.

For a Borel set $A\subset\mathbb{R}^q$ and a non-empty finite subset $\alpha$ of $\mathbb{R}^q$, we define
\begin{eqnarray}
I_{\nu,r}(A,\alpha):=\left\{\begin{array}{ll}\int_{A}d(x,\alpha)^rd\nu(x)&r>0\\
\int_{A}\log d(x,\alpha)d\nu(x)&r=0\end{array}\right..
\end{eqnarray}
In the following, we simply write $I_\nu(A,\alpha)$ for $I_{\nu,0}(A,\alpha)$ for convenience.
\subsection{A significant concern about the Voronoi partition}
Let $\nu$ be an absolutely continuous probability measure on $\mathbb{R}^q$. In \cite{Ger:79}, Gersho conjectured that for $\alpha_n\in C_{n,r}(\nu)$ and an arbitrary VP $\{P_a(\alpha_n)\}_{a\in\alpha_n}$ with respect to $\alpha_n$, we have
\begin{equation}\label{gersho}
\lim_{n\to\infty}\frac{I_{\nu,r}(P_a(\alpha_n),\{a\})}{n^{-1}e_{n,r}^r(\nu)}=1,\;\;a\in\alpha_n.
\end{equation}
This conjecture is also significant for singular  Borel probability measures.

In \cite{GL:12}, Graf, Luschgy and Pag\`{e}s proved that, for a large class of absolutely continuous probability measures $\nu$, there exist constants $A_1,A_2>0$ such that
\begin{equation}\label{g5}
A_1n^{-1}e_{n,r}^r(\nu)\leq I_{\nu,r}(P_a(\alpha_n),\{a\})\leq A_2n^{-1}e_{n,r}^r(\nu),\;\;a\in\alpha_n.
\end{equation}
This is a weak version of (\ref{gersho}). One may see \cite{Kr:13,Zhu:20} for some other related work. We remark that for general probability measures, it is very difficult even to examine whether (\ref{g5}) holds.

It is known from \cite{GL:04} that $e_{n,r}(\nu)\to e_n(\nu)$ as $r$ decreases to zero. Thus, by letting $r\to 0$ in (\ref{g5}), it is natural to conjecture that, for a Borel probability measure $\nu$, there exist some constants $B_1,B_2$ such that, for an arbitrary $\alpha_n\in C_n(\nu)$ and an arbitrary VP $(P_a(\alpha_n))_{a\in\alpha_n}$, the following holds:
\begin{equation}\label{wgc}
B_1n^{-1}\leq\min_{a\in\alpha_n}\nu(P_a(\alpha_n))\leq\max_{a\in\alpha_n}\nu(P_a(\alpha_n))\leq B_2n^{-1},\;\;a\in\alpha_n.
\end{equation}
This can be regarded as a limiting case of the weak version (\ref{g5}).
\subsection{Statement of the main result}
A Borel measure $\mu$ $\mathbb{R}^q$ is called an $s_0$-dimensional Ahlfors-David measure if there exists some $\epsilon_0>0$ such that, for every $x\in{\rm supp}(\mu)$,
\begin{equation}\label{AD}
C_1\epsilon^{s_0}\leq\mu(B(x,\epsilon))\leq C_2\epsilon^{s_0},\;\epsilon\in(0,\epsilon_0).
\end{equation}
The asymptotics of the quantization errors for Ahlfors-David measures have been studied in detail by Graf and Luschgy (cf. \cite[Theorem 12.18]{GL:00}). One can also see \cite{GL:00,Mattila:95} for various examples of such measures.

In the remaining part of the paper, we always denote by $\mu$ a probability measure satisfying (\ref{AD}). In addition, by Lemma 12.3 of \cite{GL:00}, we assume that the second inequality in (\ref{AD}) holds for all $x\in\mathbb{R}^q$ and all $\epsilon>0$. For a set $B\subset\mathbb{R}^q$, we denote the diameter of $B$ by $|B|$. We will prove
\begin{theorem}\label{mthm}
Let $\mu$ be an Ahlfors-David probability measure on $\mathbb{R}^q$ with support $K_\mu$. There exist positive constants $d_1, d_2, d_3$, such that for each $n\geq 1$, every $\alpha_n\in C_n(\mu)$ and an arbitrary VP $\{P_a(\alpha_n)\}_{a\in\alpha_n}$, we have
\[
d_1n^{-1}\leq\min_{a\in\alpha_n}\mu(P_a(\alpha_n))\leq\max_{a\in\alpha_n}\mu(P_a(\alpha_n))\leq d_2n^{-1}.
\]
Moreover, for every $a\in\alpha_n$, $P_a(\alpha_n)$ contains a ball of radius $d_3|P_a(\alpha_n)\cap K_\mu|$ which is centered at $a$.
\end{theorem}

Our main idea for the proof of Theorem \ref{mthm} is to reduce the quantization problem for $\mu$ with respect to an arbitrarily large $n$ to that for some conditional measures of $\mu$ with respect to some bounded integers, and then we apply Theorem 2.4 of \cite{GL:04} which says that a subset $\beta$ of a $k$-optimal set is ${\rm card}(\beta)$-optimal for the corresponding conditional measure of $\mu$. In order to accomplish the above-mentioned reduction, we will select some auxiliary integers and establish a characterization for $n$-optimal sets for $\mu$ with respect to the geometric mean error. In order to complete the proof of Theorem \ref{mthm} by using \cite[Theorem 2.4]{GL:04}, we will prove some weak estimates for the measures and geometrical size of elements of an optimal Voronoi partition. These results will be given in a more general context and allow us to drop an additional condition in \cite{Zhu:13} that the considered measure vanishes on every hyperplane.

Unlike the $L_r$-quantization problem, in the study of the geometric mean error, the involved integrals are usually negative and the integrands are in logarithmic form. It turns out that those methods to characterize the optimal sets in the $L_r$-quantization problem are often not applicable. For instance, let $A\supset B$ be Borel sets and $\alpha$ a non-empty finite set, we have
$I_{\nu,r}(A,\alpha)\geq I_{\nu,r}(B,\alpha)$ for $r>0$; while for $r=0$, we usually have an inequality in the reverse direction:
$I_\nu(A,\alpha)\leq I_\nu(B,\alpha)$, because the integrands are usually negative. For this reason, the arguments in the present paper are substantially different from those in \cite{Zhu:20} which are for $L_r$-quantization for $\mu$.
\section{Preliminaries}
For a probability measure $\nu$ on $\mathbb{R}^q$, we always denote the support of $\nu$ by $K_\nu$.
Let $m$ be the smallest integer with $m>2(C_1^{-1}C_2)^{1/s_0}$. Let $k_0$ be the smallest integer such that $2m^{-k_0}<\epsilon_0$. Note that, for the Ahlfors-David probability measure $\mu$, $K_\mu$ is compact. Thus, for every $k\geq k_0$, we denote by $\phi_k$ the largest number of closed balls of radii $m^{-k}$ which are pairwise disjoint and centered in $K_\mu$. We fix such $\phi_k$ closed balls and denote them by $E_{k,i},1\leq i\leq\phi_k$. We define
\[
\Omega_k:=\{(k,i):1\leq i\leq \phi_k\}.
\]
By the definition of $m$, one can show that $\phi_k<\phi_{k+1}$ by using (\ref{AD}) and the arguments in the proof of \cite[Lemma 2.1]{Zhu:20}. For $\sigma\in\Omega_k$, we denote the center of $E_\sigma$ by $c_\sigma$ and define $A_\sigma:=B(c_\sigma,2m^{-k})$. Then
we have $K_\mu\subset\bigcup_{\sigma\in\Omega_k}A_\sigma$. The following lemma is a consequence of (\ref{AD}).
\begin{lemma}\label{lem5}
There exist constants $\eta_1,\eta_2>0$, such that for every $\sigma\in\Omega_k$,
\[
\eta_1\phi_k^{-1}\leq\mu(A_\sigma)\leq\eta_2\phi_k^{-1}.
\]
\end{lemma}
\begin{proof}
Note that $K_\mu\subset\bigcup_{\sigma\in\Omega_k}A_\sigma$. By (\ref{AD}), for every $\sigma\in\Omega_k$, we have
\begin{eqnarray*}
&&1\leq\sum_{\tau\in\Omega_k}\mu(A_\tau)\leq\phi_k C_2C_1^{-1}\mu(A_\sigma);
\\&&\mu(A_\sigma)\leq C_2C_1^{-1}2^{s_0}\mu(E_\sigma).
\end{eqnarray*}
It follows that $\mu(A_\sigma)\geq C_2^{-1}C_1\phi_k^{-1}$. Because the sets $E_\tau,\tau\in\Omega_k$, are pairwise disjoint, by (\ref{AD}), we have
\[
1\geq\sum_{\tau\in\Omega_k}\mu(E_\tau)\geq\phi_k C_1C_2^{-1}\mu(E_\sigma)\geq\phi_k C_1^2C_2^{-2}2^{-s_0}\mu(A_\sigma).
\]
Hence, we have $\mu(A_\sigma)\leq (C_1^{-1}C_2)^22^{s_0}\phi_k^{-1}$. It suffices to define
\[
\eta_1:=C_2^{-1}C_1\;{\rm and}\;\eta_2:=(C_1^{-1}C_2)^22^{s_0}.
\]
\end{proof}

Let $C_1, C_2$ be as given in (\ref{AD}). We define
\begin{eqnarray}
&&\delta:=\frac{1}{16}\big(C_1C_2^{-1}\big)^{\frac{1}{s_0}};\label{g3}\\
&&\mathcal{A}_\sigma:=\{\tau\in\Omega_k: (A_\tau)_{2\delta|A_\tau|}\cap(A_\sigma)_{2\delta|A_\sigma|}\neq\emptyset\};\label{z8}\\
&&M_\sigma:={\rm card}(\mathcal{A}_\sigma);\;A_\sigma^*:=\bigcup_{\tau\in\mathcal{A}_\sigma}A_\tau,\;\;\sigma\in\Omega_k.\label{z7}
\end{eqnarray}
\begin{remark}
The number $\delta$ is defined as above for two reasons. First, $\delta$ should be small enough so that the set $E_\omega\setminus B(x_0,2^{-1}\delta|A_\omega|)$ is large enough. Second, it will be convenient for us to estimate the $\mu$-measure of a ball $B(x_0,2^{-1}\delta|A_\omega|)$ by using (\ref{AD}). One may see Lemma \ref{lem2} below for more details.
\end{remark}

For $x\in\mathbb{R}$, let $[x]$ denote the largest integer not exceeding $x$. For $t>0$ and a set $A\subset\mathbb{R}^q$, we denote the closed $t$-neighborhood of $A$ by $(A)_t$.
 \begin{remark}\label{rem0}
 Let $L_0:=[2\delta^{-1}+10]$. By estimating the volumes,  we know that for each $\sigma\in\Omega_k$, the set $(A_\sigma)_{2\delta|A_\sigma|}$ can be covered by $L_0$ closed balls of radii $2^{-1}\delta|A_\sigma|$ which are centered in $(A_\sigma)_{2\delta|A_\sigma|}$. This can be seen as follows. First, we consider the largest number of pairwise disjoint closed balls of radii $4^{-1}\delta|A_\sigma|$ which are centered in $(A_\sigma)_{2\delta|A_\sigma|}$; and then we double the radii of the balls and get a cover for $(A_\sigma)_{2\delta|A_\sigma|}$. In the remaining part of the paper, we always denote by $B_\sigma$ the set of the centers of such $L_0$ closed balls.
 \end{remark}

 Using the next lemma, we collect some basic facts regarding $A_\sigma^*$. These facts will allow us to adjust the number of prospective optimal points in $(A_\sigma)_{\delta|A_\sigma|}$ without affecting the points in $K_\mu\setminus A_\sigma^*$ unfavorably. One may apply Lemma 8 of \cite{Zhu:08} to obtain an optional proof.
\begin{lemma}\label{lem9}
Let $\sigma\in\Omega_k$ and let $\emptyset\neq\beta\subset\mathbb{R}^q$ be a finite set. Then
\begin{enumerate}
\item[(a1)] there exists an integer $M_0$ such that $M_\sigma\leq M_0$.
\item[(a2)] for every $\tau\in\Omega_k\setminus\mathcal{A}_\sigma$ and every $x\in A_\tau$, we have
\[
d(x,(\beta\setminus (A_\sigma)_{\delta|A_\sigma|})\cup B_\sigma)\leq d(x,\beta).
\]
\end{enumerate}
\end{lemma}
\begin{proof}
(a1) Note that $A_\sigma^*\subset B(c_\sigma,2(1+2\delta)|A_\sigma|)$ and $E_\tau,\tau\in\mathcal{A}_\sigma$, are pairwise disjoint. By estimating the volumes, one can see that
\[
M_\sigma (4^{-1}|A_\sigma|)^q\leq(2(1+2\delta)|A_\sigma|)^q.
\]
Hence, it is sufficient to define $M_0:=[(8(1+2\delta))^q]+1$.

(a2) Let $\tau\in\Omega_k\setminus\mathcal{A}_\sigma$ and $x\in A_\tau$. Then we have $(A_\sigma)_{2\delta|A_\sigma|}\cap(A_\tau)_{2\delta|A_\tau|}=\emptyset$. Therefore, $x\in\mathbb{R}^q\setminus(A_\sigma)_{2\delta|A_\sigma|}$. We have two cases:

 Case 1: $d(x,\beta)=d(x,\beta\setminus(A_\sigma)_{\delta|A_\sigma|})$, then (a2) is clearly true.

 Case 2: $d(x,\beta)=d(x,\beta\cap(A_\sigma)_{\delta|A_\sigma|})$. We denote the boundary of a set $B$ by $\partial B$. Note that $(A_\sigma)_{2\delta|A_\sigma|}$ is compact with non-empty interior. We may select a $z_0\in \partial(A_\sigma)_{2\delta|A_\sigma|}$ such that
 \[
 d(x,z_0)=d(x,(A_\sigma)_{2\delta|A_\sigma|})=d(x,\partial(A_\sigma)_{2\delta|A_\sigma|})
 \]
 By the definition of $B_\sigma$, there exists some $b\in B_\sigma$ such that $d(z_0,b)\leq 2^{-1}\delta|A_\sigma|$. For every $a\in \beta\cap(A_\sigma)_{\delta|A_\sigma|}$, we have
\begin{equation*}
d(x,a)\geq d(x,z_0)+\delta|A_\sigma|>d(x,z_0)+d(z_0,b)\geq d(x,b).
\end{equation*}
Hence, $d(x,\beta\cap(A_\sigma)_{\delta|A_\sigma|})\geq  d(x,b)$.
\begin{eqnarray*}
d(x,\beta)=d(x,\beta\cap(A_\sigma)_{\delta|A_\sigma|})\geq d(x,B_\sigma)\geq d(x,(\beta\setminus (A_\sigma)_{\delta|A_\sigma|})\cup B_\sigma)).
\end{eqnarray*}
This completes the proof of the lemma.
\end{proof}

\section{Weak estimates for measures and geometrical size of $P_a(\alpha_k)$}

Let $C,t>0$. We consider compactly supported measures $\nu$ satisfying
\begin{equation}\label{holder}
\sup_{x\in\mathbb{R}^q}\nu(B(x,\epsilon))\leq C\epsilon^t \;\;{\rm for\; every}\; \epsilon>0.
\end{equation}
Without loss of generality, we may assume that $C\geq1$. As in \cite{GL:04}, we write
\[
\hat{e}_k(\nu)=\log e_{k,0}(\nu)=\inf_{\alpha\in\mathcal{D}_k}I_\nu(\mathbb{R}^q,\alpha).
\]

The following lemma can be seen as an analogue of \cite[Lemma 2.1]{Zhu:20}.
\begin{lemma}\label{lem1}
Let $\nu$ be a Borel probability measure on $\mathbb{R}^q$ with compact support $K_\nu$. Assume that $|K_\nu|\leq 1$ and (\ref{holder}) is satisfied.
Then for every $k\geq 2$, there exists a real number $\zeta_k>0$, which depends on $C$ and $t$ such that
\[
\hat{e}_{k-1}(\nu)-\hat{e}_k(\nu)\geq\zeta_k.
\]
\end{lemma}
\begin{proof}
Let $\alpha_{k-1}=\{a_i\}_{i=1}^{k-1}\in C_{k-1}(\nu)$. We define
\begin{eqnarray*}
&&\delta_{k,1}:=(4C(k-1))^{-\frac{1}{t}};\;\delta_{k,2}:=(2C(k-1))^{-\frac{1}{t}};\\
&&\delta_k:=2^{-1}\min\{\delta_{k,1},\delta_{k,2}-\delta_{k,1}\}.
\end{eqnarray*}
By (\ref{holder}), we have $\nu(K_\nu\setminus\bigcup_{i=1}^{k-1}B(a_i,\delta_{k,2}))\geq 2^{-1}$.
Let $l_k:=[(2\delta_k^{-1}+2)^q]+1$. Note that $|K_\nu|\leq 1$. Hence, $K_\nu\setminus\bigcup_{i=1}^{k-1}B(a_i,\delta_{k,2})$ can be covered by $l_k$ closed balls $B_i (1\leq i\leq l_k$) of radii $\delta_k$ which are centered in the set $K_\nu\setminus\bigcup_{i=1}^{k-1}B(a_i,\delta_{k,2})$ (cf. Remark \ref{rem0}). Thus, there exists some ball $B_i$ such that $\nu(B_i)\geq (2l_k)^{-1}$. Let $b_i$ denote the center of $B_i$. Then
\begin{eqnarray*}
\hat{e}_{k-1}(\nu)-\hat{e}_k(\nu)&\geq& I_\nu(\mathbb{R}^q,\alpha_{k-1})-I_\nu(\mathbb{R}^q,\alpha_{k-1}\cup\{b_i\})\\
&\geq&I_\nu(B_i,\alpha_{k-1})-I_\nu(B_i,\alpha_{k-1}\cup\{b_i\})\\&\geq&I_\nu(B_i,\alpha_{k-1})-I_\nu(B_i,\{b_i\})\\&\geq&\nu(B_i)(\log\delta_{k,1}-\log\delta_k)
\\&\geq&(2l_k)^{-1}\log 2.
\end{eqnarray*}
The proof of the lemma is complete by defining $\zeta_k:=(2l_k)^{-1}\log 2$.
\end{proof}

Using the next lemma, we establish a lower bound for the $\nu$-measure of the elements of a VP with respect to a $k$-optimal set for $\nu$ of order zero.
\begin{lemma}\label{lem11}
 Assume that the hypothesis of Lemma \ref{lem1} is satisfied. For each $k\geq 2$, there exists a positive real number $\underline{d}_k$ such that, for every $\alpha_k\in C_k(\nu)$ and an arbitrary VP $\{P_a(\alpha_k)\}_{a\in\alpha_k}$ with respect to $\alpha_k$, we have
\[
\min_{a\in\alpha_k}\nu(P_a(\alpha_k))\geq \underline{d}_k.
\]
\end{lemma}
\begin{proof}
Let $\alpha_k\in C_k(\nu)$ and $\{P_a(\alpha_k)\}_{a\in\alpha_k}$ a VP. Let $a\in\alpha_k$ be fixed. By Theorem 2.4 of \cite{GL:04}, $\nu(P_a(\alpha_k))>0$. We define $\beta:=\alpha_k\setminus\{a\}$. Then for every $x\in\bigcup_{b\in\beta}P_b(\alpha_k)$, we have $d(x,\beta)=d(x,\alpha_k)$. Thus,
\begin{eqnarray}\label{s10}
\hat{e}_{k-1}(\nu)-\hat{e}_k(\nu)&\leq& I_\nu(\mathbb{R}^q,\beta)-I_\nu(\mathbb{R}^q,\alpha_k)\nonumber\\&=&I_\nu(P_a(\alpha_k),\beta)-I_\nu(P_a(\alpha_k),\{a\}).
\end{eqnarray}
Note that $\sup_{x\in K_\nu}d(x,\alpha_k)\leq 2|K_\nu|\leq 2$ (cf. \cite[Lemma 5.8]{GL:04}). Therefore for every $x\in P_a(\alpha_k)\cap K_\nu$, we have,
$d(x,\beta)\leq 3|K_\nu|\leq 3$. It follows that
\begin{equation}\label{s8}
I_\nu(P_a(\alpha_k),\beta)\leq\nu(P_a(\alpha_k))\log 3.
\end{equation}
Now by \cite[Lemma 3.4]{GL:04}, we have
\begin{equation}\label{s9}
I_\nu(P_a(\alpha_k),\{a\})\geq\frac{1}{t}\big(\nu(P_a(\alpha_k))\log\nu(P_a(\alpha_k))-C\nu(P_a(\alpha_k))\big).
\end{equation}
We define $h(x):=-x\log x$ for $x>0$. Then $h(x)\to 0$ as $x$ decreases to zero. Thus,
there exists some $\eta_k>0$ such that $0<x<\eta_k$ implies $-x\log x<2^{-1}t\zeta_k$. Therefore, if $\mu(P_a(\alpha_k))<\eta_k$, using Lemma \ref{lem1} and (\ref{s10})-(\ref{s9}), we deduce
\begin{eqnarray*}
\zeta_k\leq\hat{e}_{k-1}(\nu)-\hat{e}_k(\nu)\leq\nu(P_a(\alpha_k))\log 3+\frac{1}{2}\zeta_k+\frac{C}{t}\nu(P_a(\alpha_k)).
\end{eqnarray*}
Thus, we obtain $\nu(P_a(\alpha_k))\geq 2^{-1}(\log3+Ct^{-1})^{-1}\zeta_k$. It suffices to define
\[
\underline{d}_k:=\min\big\{\eta_k,2^{-1}(\log3+Ct^{-1})^{-1}\zeta_k\big\}.
\]
\end{proof}

Next, we establish an upper bound for $\nu(P_a(\alpha_k)),a\in\alpha_k$.
\begin{lemma}\label{l1}
 Assume that the hypothesis of Lemma \ref{lem1} is satisfied. For each $k\geq 1$, there exists a positive real number $\overline{d}_k$ such that, for every $\alpha_k\in C_k(\nu)$ and an arbitrary VP $\{P_a(\alpha_k)\}_{a\in\alpha_k}$ with respect to $\alpha_k$, we have
\[
\max_{a\in\alpha_k}\nu(P_a(\alpha_k))\leq \overline{d}_k.
\]
\end{lemma}
\begin{proof}
Let $\alpha_k\in C_k(\nu)$ and $a\in\alpha_k$. Set
\[
\delta_{k,3}:=(2C)^{-\frac{1}{t}}\nu(P_a(\alpha_k))^{\frac{1}{t}}.
\]
Then we have $\delta_{k,3}<1$, since $C\geq 1$. By (\ref{holder}), we deduce that
\[
\nu(B(a,\delta_{k,3}))\leq C\delta_{k,3}^t=2^{-1}\nu(P_a(\alpha_k)).
\]
It follows that
\begin{equation}\label{z1}
\nu(P_a(\alpha_k)\setminus B(a,\delta_{k,3}))\geq\nu(P_a(\alpha_k))-\nu(B(a,\delta_{k,3}))\geq\frac{1}{2}\nu(P_a(\alpha_k)).
\end{equation}
Let $N_k:=[(8\delta_{k,3}^{-1})^q]+3$. One can easily see that
\begin{equation}\label{zsz1}
N_k\leq (16\delta_{k,3}^{-1})^{q}\;\; {\rm and}\;\;N_k^{-1}\geq16^{-q}\delta_{k,3}^q.
\end{equation}
Note that $|P_a(\alpha_k)\cap K_\nu|\leq |K_\nu|\leq 1$. By estimating volumes, one can see that
\[
(P_a(\alpha_k)\cap K_\nu)\setminus B(a,\delta_{k,3})
 \]
 can be covered by $N_k$ closed balls $B_i (1\leq i\leq N_k)$ of radii $4^{-1}\delta_{k,3}$ which are centered in $(P_a(\alpha_k)\cap K_\nu)\setminus B(a,\delta_{k,3})$. Thus, by (\ref{z1}) and (\ref{zsz1}), there exists some ball $B_i$ such that
\begin{eqnarray}\label{z3}
\nu(B_i\cap P_a(\alpha_k))&\geq& \frac{1}{2N_k}\nu(P_a(\alpha_k)\geq32^{-q}\delta_{k,3}^q\nu(P_a(\alpha_k))\nonumber
\\&\geq&32^{-q}(2C)^{-\frac{q}{t}}\nu(P_a(\alpha_k))^{1+\frac{q}{t}}\nonumber\\
&=:&D_1\nu(P_a(\alpha_k))^{1+\frac{q}{t}}.
\end{eqnarray}
Now we define $\beta:=\alpha_k\cup\{b_i\}$. Then we have the following estimate:
\begin{eqnarray}\label{z2}
\hat{e}_k(\nu)-\hat{e}_{k+1}(\nu)&\geq& I_\nu(\mathbb{R}^q,\alpha_k)-I_\nu(\mathbb{R}^q,\beta)\nonumber\\
&\geq& I_\nu(B_i\cap P_a(\alpha_k),\alpha_k)-I_\nu(B_i\cap P_a(\alpha_k),\beta)\nonumber\\&\geq& I_\nu(B_i\cap P_a(\alpha_k),\{a\})-I_\nu(B_i\cap P_a(\alpha_k),\{b_i\})
\end{eqnarray}
By the definition of $B_i$, for every $x\in B_i$, we have
\begin{equation}\label{g4}
d(x,a)\geq \frac{1}{2}\delta_{k,3},\;\;d(x,b_i)\leq \frac{1}{4}\delta_{k,3}.
\end{equation}
Now by the proof of \cite[Lemma 5.8]{GL:04}, for every $n\geq 1$, we have
\[
\hat{e}_n(\nu)-\hat{e}_{n+1}(\nu)\leq\frac{1}{n+1}\log3+C^{\frac{1}{2}}\frac{2}{t}\big(\frac{1}{n+1}\big)^{1/2}=:\chi_n.
\]
Using  this and (\ref{z3})-(\ref{g4}), we deduce
\[
\chi_k\geq\hat{e}_k(\nu)-\hat{e}_{k+1}(\nu)\geq\nu(B_i\cap P_a(\alpha_k))\log2\geq D_1\log 2\;(\nu(P_a(\alpha_k))^{1+\frac{q}{t}}.
\]
The proof of the lemma is complete by defining $\overline{d}_k:=\big(\chi_k (D_1\log 2)^{-1}\big)^{\frac{t}{t+q}}$.

\end{proof}

We end this section with an estimate for the geometrical size of the elements of a VP with respect to a $k$-optimal set $\alpha_k\in C_k(\nu)$.
\begin{lemma}\label{l2}
 Assume that the hypothesis of Lemma \ref{lem1} is satisfied. For each $k\geq 2$, there exists a number $g_k>0$ such that, for every $\alpha_k\in C_k(\nu)$ and every VP $\{P_a(\alpha_k)\}_{a\in\alpha_k}$ with respect to $\alpha_k$ and every $a\in\alpha_k$, we have, $P_a(\alpha_k)$ contains a closed ball of radius $g_k|P_a(\alpha_k)\cap K_\nu|$ which is centered at $a$.
\end{lemma}
\begin{proof}
Let $d_H$ denote the Hausdorff metric. We define $\phi: C_k(\nu)\mapsto\mathbb{R}$ by:
\[
\phi(\alpha_k):=\min_{a\in\alpha_k}\min_{b\in\alpha_k\setminus\{a\}}d(a,b).
\]
We first show that $\phi$ is continuous on $C_k(\nu)$. To see this, it is sufficient to consider an arbitrary accumulation point  (if any) $\alpha_k=\{a_i\}_{i=1}^k$ of $C_k(\nu)$. Assume that
$\beta_{k,n}=\{b_{i,n}\}_{i=1}^k\in C_k(\nu)$ and $d_H(\beta_{k,n},\alpha_k)\to 0$ as $n\to\infty$. Without loss of generality, we assume that $\phi(\alpha_k)=d(a_1,a_2)$. Let $\eta_0:=4^{-1}\phi(\alpha_k)$. Then for every $\epsilon\in (0,\eta_0)$, there exists some $N_0\geq 1$, such that for all $n\geq N_0$, we have $d_H(\beta_{k,n},\alpha_k)<\epsilon$. Thus, for every $1\leq i\leq k$, there exists a unique $1\leq j(i)\leq k$ such that $d(b_{j(i),n},a_i)<\epsilon$. Thus, we rewrite $\beta_{k,n}$ as $\{b_{j(i),n}\}_{i=1}^k$. For $1\leq i_1\neq i_2\leq k$, by the triangle inequality, we have
\begin{eqnarray}\label{ss1}
d(b_{j(i_1),n},b_{j(i_2),n})&\geq& d(a_{i_1},a_{i_2})-d(a_{i_1},b_{j(i_1),n})-d(a_{i_2},b_{j(i_2),n})\nonumber\\
&\geq&d(a_1,a_2)-2\epsilon.
\end{eqnarray}
It follows that $\phi(\beta_{k,n})\geq \phi(\alpha_k)-2\epsilon$ for every $n\geq N_0$. Also, we have
\begin{eqnarray}\label{ss2}
\phi(\beta_{k,n})&\leq& d(b_{j(1),n},b_{j(2),n})\nonumber\\&\leq& d(b_{j(1),n},a_1)+d(a_1,a_2)+d(a_2,b_{j(2),n})\nonumber\\&<&\phi(\alpha_k)+2\epsilon.
\end{eqnarray}
From (\ref{ss1}) and (\ref{ss2}), we obtain that $|\phi(\beta_{k,n})-\phi(\alpha_k)|<2\epsilon$ for every $n\geq N_0$. It follows that $\phi(\beta_{k,n})\to\phi(\alpha_k)$ as $n\to\infty$. Thus, $\phi$ is continuous on $C_k(\nu)$.

By \cite[Theorem 2.5]{GL:04}, $C_k(\nu)$ is $d_H$-compact. Thus, by the continuity of $\phi$, there exist some $\alpha_{k,1}\in C_k(\nu)$, such that
\[
\underline{\lambda}_k(\nu):=\min_{\alpha\in C_k(\nu)}\phi(\alpha)=\phi(\alpha_{k,1}).
\]
Clearly, we have $\underline{\lambda}_k(\nu)>0$. Now let $\alpha_k=\{a_i\}_{i=1}^k$ be an arbitrary $k$-optimal set for $\nu$ and $\{P_a(\alpha_k)\}_{a\in\alpha_k}$ an arbitrary VP with respect to $\alpha_k$. Then
\begin{eqnarray*}
B(a,3^{-1}\underline{\lambda}_k(\nu))\subset P_a(\alpha_k)\;{\rm and}\;|P_a(\alpha_k)\cap K_\nu|\leq|K_\nu|\leq 1.
\end{eqnarray*}

Let $\eta_k,\zeta_k$ be as defined in the preceding lemmas. Next, we establish a lower bound for $\underline{\lambda}_k(\nu)$ in terms of $\eta_k$ and $\zeta_k$ which depend only on $C,t,k,q$. Set
\[
B_k:=\min\big\{\frac{t\zeta_k}{4C},\eta_k\big\};\;\epsilon_k:=\min\big\{\big(\frac{B_k}{2C}\big)^{1/t},2^{-1}\big\};
\;s_k:=\frac{1}{2}(e^{\frac{\zeta_k}{4}}-1)\epsilon_k.
\]
Note that $\zeta_k<1$, we have that $s_k<\epsilon_k$. We are going to show that $\underline{\lambda}_k(\nu)\geq s_k$. Suppose that $\underline{\lambda}_k(\nu)=d(a_1,a_2)<s_k$, we deduce a contradiction. Write
\[
A_{\epsilon_k}:=B(a_1,\epsilon_k)\cap P_{a_1}(\alpha_k);\;\;\beta:=\alpha_k\setminus\{a_1\}.
\]
Then we have $\nu(A_{\epsilon_k})\leq C\epsilon_k^t<\min\{\eta_k,(4C)^{-1}t\zeta_k\}$. By the proof of Lemma \ref{lem11}, we know that $-\frac{1}{t}(\nu(A_{\epsilon_k})\log \nu(A_{\epsilon_k})<2^{-1}\zeta_k$. Further, one can easily see that for every $b\in\beta$ and $x\in P_b(\alpha_k)$, we have $d(x,\alpha_k)=d(x,b)=d(x,\beta)$. Thus,
\begin{eqnarray}\label{append1}
\hat{e}_{k-1}(\nu)-\hat{e}_k(\nu)&\leq& I_\nu(\mathbb{R}^q,\beta)-I_\nu(\mathbb{R}^q,\alpha_k)\nonumber\\&=&I_\nu(P_{a_1}(\alpha_k),\beta)-I_\nu(P_{a_1}(\alpha_k),\alpha_k).
\end{eqnarray}
Note that for $x\in A_{\epsilon_k}$, we have $d(x,\alpha_k)=d(x,a_1)$ and
\[
d(x,\beta)\leq d(x,a_2)\leq d(x,a_1)+d(a_1,a_2)<\epsilon_k+s_k<2\epsilon_k<1.
\]
Using this and \cite[Lemma 3.6]{GL:04}, we deduce
\begin{eqnarray*}
\Delta_1:&=&I_\nu(A_{\epsilon_k},\beta)-I_\nu(A_{\epsilon_k},\alpha_k)\\&\leq&I_\nu(A_{\epsilon_k},\{a_2\})-I_\nu(A_{\epsilon_k},\{a_1\})
\\&\leq&\nu(A_{\epsilon_k})\log(2\epsilon_k)-\frac{1}{t}\big(\nu(A_{\epsilon_k})\log \nu(A_{\epsilon_k})-C\nu(A_{\epsilon_k})\big)\\
&\leq&-\frac{1}{t}\nu(A_{\epsilon_k})\log \nu(A_{\epsilon_k})+\frac{C}{t}\nu(A_{\epsilon_k})\\
&<&2^{-1}\zeta_k+4^{-1}\zeta_k.
\end{eqnarray*}
For every $x\in P_{a_1}(\alpha_k)\setminus A_{\epsilon_k}=:B_{\epsilon_k}$, we have
\[
d(x,\alpha_k)=d(x,a_1)>\epsilon_k,\;d(x,\beta)\leq d(x,a_2)<d(x,a_1)+s_k.
\]
From this, we deduce that
\[
\frac{d(x,\beta)}{d(x,\alpha_k)}\leq\frac{d(x,a_2)}{d(x,a_1)}\leq \frac{d(x,a_1)+s_k}{d(x,a_1)}=1+\frac{s_k}{\epsilon_k}<e^{\frac{\zeta_k}{4}}.
\]
By the preceding inequality and the fact that $\nu(B_{\epsilon_k})<1$, we obtain
\begin{eqnarray*}
\Delta_2:=I_\nu(B_{\epsilon_k},\beta)-I_\nu(B_{\epsilon_k},\alpha_k)\leq\nu(B_{\epsilon_k})\log e^{\frac{\zeta_k}{4}}<\frac{\zeta_k}{4}.
\end{eqnarray*}
From this and (\ref{append1}), we deduce that $\hat{e}_{k-1}(\nu)-\hat{e}_k(\nu)\leq\Delta_1+\Delta_2<\zeta_k$, contradicting Lemma \ref{lem1}.
Thus, $\underline{\lambda}_k(\nu)\geq s_k$ and the proof of the lemma is complete by defining $g_k=3^{-1}s_k$.
\end{proof}
\section{Auxiliary measures and auxiliary integers}

\subsection{Some subsets of $A_\omega^*$ and auxiliary measures}
For a finite subset $\alpha$ of $\mathbb{R}^q$, let $W(a|\alpha),a\in\alpha$, be as defined in (\ref{vregion}).
Let $\delta$ be as defined in (\ref{g3}). Let $\omega,\sigma\in\Omega_k$ with $\sigma\neq\omega$. Let $x_0\in A_\omega\cap K_\mu$. The following three types of subsets of $A_\omega^*$ will be considered in the characterization for the optimal sets for $\mu$:
\begin{eqnarray*}
&&D_{\omega,1}^{(\alpha)}:=E_\omega\cup\bigg(\bigcup_{a\in\alpha\cap (A_\omega)_{\delta|A_\omega|}}(W(a|\alpha)\cap A_\omega^*)\bigg)\setminus B(x_0,2^{-1}\delta|A_\omega|);\\
&&D_{\omega,2}^{(\alpha)}(\sigma):=E_\omega\cup\bigg(\bigcup_{a\in\alpha\cap (A_\omega)_{\delta|A_\omega|}}(W(a|\alpha)\cap A_\omega^*)\bigg)\setminus E_\sigma;\\
&&D_{\omega,3}^{(\alpha)}:=E_\omega\cup\bigg(\bigcup_{a\in\alpha\cap (A_\omega)_{\delta|A_\omega|}}(W(a|\alpha)\cap A_\omega^*)\bigg).
\end{eqnarray*}
If no confusion arises, we write $D_{\omega,i}$ for $D_{\omega,i}^{(\alpha)}$ and write $D_{\omega,2}$ for $D_{\omega,2}^{(\alpha)}(\sigma)$.

For $\omega\in\Omega_k$, recall that $c_\omega$ is the center of $E_\omega$. We define
\[
D_{\omega,4}:=B(c_\omega,(2^{-1}-\delta)|E_\omega|)\subset E_\omega.
\]

For $1\leq i\leq 4$, let $\mu(\cdot|D_{\omega,i})$ denote the conditional measure of $\mu$ on $D_{\omega,i}$:
\begin{equation}\label{auxmeasure}
\mu(\cdot|D_{\omega,i})(A)=\frac{\mu(A\cap D_{\omega,i})}{\mu(D_{\omega,i})},\;A\;{\rm is\;a\; Borel\;set}.
\end{equation}
Let $f_{D_{\omega,i}}$ be a similarity mapping of similarity ratio $|D_{\omega,i}|$ and define
\begin{equation}
\nu_{D_{\omega,i}}:=\mu(\cdot|D_{\omega,i})\circ f_{D_{\omega,i}},\;K_{\nu_{D_{\omega,i}}}:={\rm supp}(\nu_{D_{\omega,i}}).
\end{equation}
Then $\nu_{D_{\omega,i}}$  is a probability measure satisfying $|K_{\nu_{D_{\omega,i}}}|\leq 1$.

In a similar manner, we define the measures $\nu_{E_\sigma},\sigma\in\Omega_k$. We have

\begin{lemma}\label{lem2}
There exists a constant $C$ such that, for $B\in \{D_{\omega,i}\}_{i=1}^4\cup\{E_\omega\}$, we have $\sup_{x\in\mathbb{R}^q}\nu_B(B(x,\epsilon))\leq C\epsilon^{s_0}$ for every $\epsilon>0$.
\end{lemma}
\begin{proof}
By the definition, $D_{\omega,i}\subset A_\omega^*$ and $A_\omega^*\subset B(c_\omega, 2(1+2\delta)|A_\omega|)$. Hence,
\begin{equation}\label{g6}
|D_{\omega,i}|\leq |A_\omega^*|\leq 4(1+2\delta)|A_\omega|,\;1\leq i\leq 3.
\end{equation}
Since the diameter of $B(x_0,2^{-1}\delta|A_\omega|)$ is equal to $\delta|A_\omega|$, we have
\begin{equation}\label{zz1}
(1-2\delta)|E_\omega|\leq\bigg|E_\omega\setminus B\big(x_0,2^{-1}\delta|A_\omega|\big)\bigg|\leq|D_{\omega,1}|\leq4(1+2\delta)|A_\omega|.
\end{equation}
It follows that $|E_\omega|=2^{-1}|A_\omega|\geq(8(1+2\delta))^{-1}|D_{\omega,1}|$. This and (\ref{AD}) yield that
\begin{eqnarray*}
\mu(D_{\omega,1})&\geq&\mu(E_\omega\setminus B\big(x_0,2^{-1}\delta|A_\omega|)\\&\geq&\mu(E_\omega)-\mu(B\big(x_0,2^{-1}\delta|A_\omega|))\\&\geq& C_1(2^{-1}|E_\omega|)^{s_0}-C_2(2^{-1}\delta|A_\omega|)^{s_0}\\&\geq&C_1(2^{-1}|E_\omega|)^{s_0}-C_1(16^{-1}|E_\omega|)^{s_0}
\\&=&C_1(2^{-s_0}-16^{-s_0})|E_\omega|^{s_0}\\&\geq&C_1(2^{-s_0}-16^{-s_0})(8(1+2\delta))^{-s_0}|D_{\omega,1}|^{s_0}.
\end{eqnarray*}
We write $\xi_1:=C_1(2^{-s_0}-16^{-s_0})(8(1+2\delta))^{-s_0}$. On the other hand, by (\ref{zz1}),
\begin{eqnarray*}
\mu(D_{\omega,1})&\leq&\mu(A_\omega^*)\leq C_2(2(1+2\delta)|A_\omega|)^{s_0}\\&\leq& C_2(4(1+2\delta))^{s_0}(1-2\delta)^{-s_0}|D_{\omega,1}|^{s_0}=:\xi_2|D_{\omega,1}|)^{s_0}.
\end{eqnarray*}
Note that for distinct words $\sigma,\omega\in\Omega_k$, we have $E_\sigma\cap E_\omega=\emptyset$. Thus, for $i=2,3$, we have $E_\omega\subset D_{\omega,i}\subset A_\omega^*$. Using these facts and (\ref{g6}), we deduce
\begin{eqnarray*}
\mu(D_{\omega,i})&\leq&\mu(A_\omega^*)\leq C_2(2(1+2\delta)|A_\omega|)^{s_0}\\&\leq& C_24^{s_0}(1+2\delta)^{s_0}|E_\omega|^{s_0}\leq C_24^{s_0}(1+2\delta)^{s_0}|D_{\omega,i}|^{s_0}=:\xi_3|D_{\omega,i}|^{s_0};\\
\mu(D_{\omega,i})&\geq&\mu(E_\omega)\geq C_1(2^{-1}|E_\omega|)^{s_0}\geq C_1 4^{-s_0}|A_\omega|^{s_0}\\
&\geq& C_1 4^{-s_0}(4(1+2\delta))^{-s_0}|D_{\omega,i}|^{s_0}=:\xi_4|D_{\omega,i}|^{s_0}.
\end{eqnarray*}
For every $\omega\in\Omega_k$, we have
\begin{eqnarray*}
C_1(2^{-1}|E_\omega|)^{s_0})\leq\mu(E_\omega)\leq C_2(2^{-1}|E_\omega|)^{s_0});\\
C_1(2^{-1}|D_{\omega,4}|)^{s_0})\leq\mu(D_{\omega,4})\leq C_2(2^{-1}|D_{\omega,4}|)^{s_0})
\end{eqnarray*}
We define $\xi:=\max\{\xi_1^{-1},\xi_4^{-1},\xi_2,\xi_3\}$. Then by the above analysis, we obtain
\begin{equation}\label{xi}
\xi^{-1}|B|^{s_0}\leq\mu(B)\leq \xi|B|^{s_0}.
\end{equation}
for $B\in\{D_{\omega,i}\}_{i=1}^4\cup\{E_\omega\}$. Thus, the lemma follows from \cite[Lemma 2.5]{Zhu:20}.
\end{proof}
\begin{remark}\label{rem2}
Let $\xi$ be as defined in (\ref{xi}). For $1\leq i\leq 3$, we have
\begin{eqnarray*}
\mu(D_{\omega,i})&\leq&\xi|D_{\omega,i}|^{s_0}\\&\leq& \xi|A_\omega^*|^{s_0}\leq \xi(4(1+2\delta)|A_\omega|)^{s_0}\\&\leq&
 \xi(8(1+2\delta))^{s_0}|E_\omega|^{s_0}\\&\leq&\xi(8(1+2\delta))^{s_0}C_1^{-1}(2^{-1}-\delta)^{-s_0}\min_{\sigma\in\Omega_k}\mu(D_{\sigma,4}).
\end{eqnarray*}
Let $\zeta:=\xi(8(1+2\delta))^{s_0}C_1^{-1}(2^{-1}-\delta)^{-s_0}$. Then for every $\sigma\in\Omega_k$, we have
\[
\max_{1\leq i\leq 3}\mu(D_{\omega,i})\leq\zeta\mu(D_{\sigma,4})\leq \zeta\mu(E_\sigma).
\]
\end{remark}

In the following we denote by $f_B$ the similarity mapping in the definition of the measure $\nu_B$. The subsequent two lemmas will be very important for the characterization for the optimal sets. One of them is a consequence of the definition of the auxiliary measures $\nu_B$, and the other is based on Lemma \ref{lem9}.
\begin{lemma}\label{l3}
Let $B\in\{D_{\omega,i}\}_{i=1}^4\cup\{E_\omega\}$. Let $\alpha$ be a non-empty finite subset of $\mathbb{R}^q$ with ${\rm card}(\alpha)=l_\alpha$. Then
$I_\mu(B,\alpha)\geq\mu(B)\log |B|+\mu(B)\hat{e}_{l_\alpha}(\nu_B)$, and equality holds if $f_B^{-1}(\alpha)\in C_{l_\alpha}(\nu_B)$.
\end{lemma}
\begin{proof}
By the definition of $\nu_B$ (cf. (\ref{auxmeasure})), we have
\begin{eqnarray*}
I_\mu(B,\alpha)&=&\int_{B}\log d(x,\alpha)d\mu(x)\\
&=&\mu(B)\int_{B}\log d(x,\alpha)d\mu(\cdot|B)(x)\\&=&\mu(B)\int_{B}\log d(x,\alpha)d\nu_B\circ f_B^{-1}(x)\\
&=&\mu(B)\log |B|+\mu(B)\int_{f_B^{-1}(B)}\log d(x,f^{-1}_B(\alpha))d\nu_B(x)\\&\geq&\mu(B)\log |B|+\mu(B)\hat{e}_{l_\alpha}(\nu_B).
\end{eqnarray*}
This completes the proof of the lemma.
\end{proof}

Let $\omega,\tau,\sigma\in\Omega_k$ with $\sigma\neq\omega$ and $\tau\in\Omega_k\setminus\mathcal{A}_\omega$. Let $D_{\omega,i},1\leq i\leq 3$, be as defined above. We write
\[
F_{\omega,i}=\left\{\begin{array}{ll}D_{\omega,1}\cup B(x_0,2^{-1}\delta|A_\omega|)&\;\;\;\;i=1\\
D_{\omega,2}\cup E_\sigma&\;\;\;\;i=2
\\D_{\omega,3}\cup D_{\tau,4}&\;\;\;\;i=3\end{array}\right..
\]
\begin{lemma}\label{compare2}
Let $\alpha,\gamma$ be non-empty finite subsets of $\mathbb{R}^q$. Let $B_\omega$ be as defined in Remark \ref{rem0}. We define
$\beta:=(\alpha\setminus(A_\omega)_{\delta|A_\omega|})\cup B_\omega\cup\gamma$. Then
\[
I_\mu(\mathbb{R}^q\setminus F_{\omega,i},\beta)\leq I_\mu(\mathbb{R}^q\setminus F_{\omega,i},\alpha),\;\;1\leq i\leq 3.
\]
\end{lemma}
\begin{proof}
Let $1\leq i\leq 3$ be fixed. By the definition of $D_{\omega,i}$ and $F_{\omega,i}$, we have
\[
\{x\in A_\omega^*: d(x,\alpha)=d(x,\alpha\cap(A_\sigma)_{\delta|A_\sigma|})\}\subset F_{\omega,i}.
\]
Therefore, for every $x\in \mathbb{R}^q\setminus F_{\omega,i}$, we have the following two cases:
\begin{enumerate}
\item[(b1)] $x\in A_\omega^*$ and $d(x,\alpha)=d(x,\alpha\setminus(A_\sigma)_{\delta|A_\sigma|})$; then clearly $d(x,\beta)\leq d(x,\alpha)$;
\item[(b2)] $x\in K_\mu\setminus A_\omega^*$. Note that $K_\mu\subset\bigcup_{\tau\in\Omega_k}A_\tau$. Thus, $x\in A_\tau$ for some $\tau\in\Omega_k\setminus\mathcal{A}_\omega$. By Lemma \ref{lem9} (a2), we also have, $d(x,\beta)\leq d(x,\alpha)$.
\end{enumerate}
Thus, $d(x,\beta)\leq d(x,\alpha)$ for every $x\in \mathbb{R}^q\setminus F_{\omega,i}$, which implies the lemma.
\end{proof}
\subsection{Selection of some auxiliary integers}

Let $L_0$ be as defined in Remark \ref{rem0}. We define
\begin{eqnarray*}
L_1:=[(2\delta^{-1}+1)^q]+1,\;\;L_2:=6^q;\;\;n_0:=[(4\delta^{-1}+1)^q]+1.
\end{eqnarray*}

\begin{remark}\label{rem4}
By estimating the volumes, one can see the following facts:
\begin{enumerate}
\item[(c1)] for each $\sigma\in\Omega_k$, the set $E_\sigma$ can be covered by $L_1$ closed balls of radii $2^{-1}\delta|E_\sigma|$ which are centered in $E_\sigma$. We denote by $\gamma_{E_\sigma}$ the set of the centers of such $L_1$ balls.

\item[(c2)] for $x\in A_\sigma\cap K_\mu$, the ball $B(x,2^{-1}\delta|A_\sigma|)$ can always be covered by $L_2$ closed balls of radii $4^{-1}\delta|A_\sigma|$ which are centered in $B(x,2^{-1}\delta|A_\sigma|)$. We will denote by $G_x$ the set of the centers of such $L_2$ closed balls.
\item[(c3)]$A_\sigma$ can be covered by $n_0$ closed balls of radii $4^{-1}\delta|A_\sigma|$ which are centered in $A_\sigma$. We denote by $H_\sigma$ the set of the centers of such $n_0$ balls.
\end{enumerate}
\end{remark}

With the above preparations, we are able to define an integer $n_1$ which will be used to give a lower estimate for the number of optimal points in $(A_\sigma)_{\delta|A_\sigma|}$.
\begin{lemma}\label{lem10}
Let $\zeta$ be as defined in Remark \ref{rem2}. There exists a smallest integer $n_1$ such that for every $\omega\in\Omega_k$ and $n\geq n_1$, we have
\[
\hat{e}_{n-L_0-L_2}(\nu_{D_{\omega,2}})-\hat{e}_{n+L_1}(\nu_{D_{\omega,2}})<\zeta^{-1}C_1C_2^{-1}\delta^{s_0}\log 2.
\]
\end{lemma}
\begin{proof}
By Lemma \ref{lem2} and \cite[Lemma 5.8]{GL:04}, we deduce,
\begin{eqnarray*}
&&\lim_{n\to\infty}(\hat{e}_{n-L_0-L_2}(\nu_{D_{\omega,2}})-\hat{e}_{n+L_1}(\nu_{D_{\omega,2}}))\\
&&\;\;\;\;\;\;\;\;=\sum_{h=-(L_0+L_2)}^{L_1-1}\lim_{n\to\infty}(\hat{e}_{n+h}(\nu_{D_{\omega,2}})-\hat{e}_{n+h+1}(\nu_{D_{\omega,2}}))=0.
\end{eqnarray*}
This implies the lemma.
\end{proof}

By \cite[Lemma 2.1]{Zhu:20}, there exists an integer $N$ such that $\phi_{k+1}\leq N\phi_k$. Next, we select three more integers $n_2,n_3, n_4$. These integers will be used to establish an upper bound for the number of optimal points in $(A_\sigma)_{\delta|A_\sigma|}$.

\begin{lemma}\label{lem6}
Let $\zeta$ and $M_0$ be as defined in Remark \ref{rem2} and Lemma \ref{lem9}. Then

 (d1) there exists a smallest integer $n_2>n_1+L_0+L_1$, such that for every $n\geq n_2$, $\sigma,\omega\in\Omega_k$, the following holds:
\[
\hat{e}_{n-L_0-n_1-L_1}(\nu_{D_{\omega,2}})-\hat{e}_{n+L_1}(\nu_{D_{\omega,2}})
<\zeta^{-1}\big(\hat{e}_{n_1+L_1-1}(\nu_{E_\sigma})-\hat{e}_{n_1+L_1}(\nu_{E_\sigma})\big);
\]

(d2) let $n_3:=(n_2+n_0)N$; there exists a smallest integer $n_4>M_0n_3+L_0+L_1$, such that for $n\geq n_4$ and every pair $\sigma,\omega\in\Omega_k$, the following holds:
\[
\hat{e}_{n-L_0-n_3-L_1}(\nu_{D_{\sigma,3}})-\hat{e}_{n+L_1}(\nu_{D_{\sigma,3}})
<\zeta^{-1}\big(\hat{e}_{n_3+L_1-1}(\nu_{D_{\omega,4}})-\hat{e}_{n_3+L_1}(\nu_{D_{\omega,4}})\big).
\]
\end{lemma}
\begin{proof}
This is a consequence of \cite[Lemma 5.8]{GL:04} and Lemmas \ref{lem1} and \ref{lem2}.
\end{proof}

\section{A characterization for the $n$-optimal sets}

Our first lemma in this section is analogous to \cite[Lemma 4.1]{Zhu:20}.
\begin{lemma}\label{lc}
We have $L_c:={\rm card}(\alpha_n\setminus\bigcup_{\sigma\in\Omega_k}(A_\sigma)_{\delta|A_\sigma|})\leq n_0\phi_k$.
\end{lemma}
\begin{proof}
Assume that $L_c>n_0\phi_k$. Let $H_\sigma$ be as defined in Remark \ref{rem4} (c3). Set
\[
\beta:=\bigg(\alpha_n\cap \bigcup_{\sigma\in\Omega_k}(A_\sigma)_{\delta|A_\sigma|}\bigg)\cup\bigg(\bigcup_{\sigma\in\Omega_k}H_\sigma\bigg).
\]
Then ${\rm card}(\beta)\leq n$. For $x\in K_\mu\subset\bigcup_{\sigma\in\Omega_k}A_\sigma$, we have $d(x,\beta)\leq d(x,\alpha_n)$. We choose an arbitrary $x\in K_\mu  $ with $d(x,\alpha_n)=d(x,\alpha_n\setminus\bigcup_{\sigma\in\Omega_k}(A_\sigma)_{\delta|A_\sigma|})$. Then we have, $d(x,\alpha_n)>\delta|A_\sigma|$ for some $\sigma\in\Omega_k$. Thus, for every $y\in B_x:=B(x,4^{-1}\delta|A_\sigma|)$, we have $d(y,\alpha_n)\geq \frac{3}{4}\delta|A_\sigma|$. Hence,
\begin{eqnarray*}
I_\mu(\mathbb{R}^q,\alpha_n)-I_\mu(\mathbb{R}^q,\beta)&\geq& I_\mu(B_x,\alpha_n)-I_\mu(B_x,\beta)\\&\geq&\mu(B_x)\big(\log(\frac{3}{4}\delta|A_\sigma|)-\log(\frac{1}{4}\delta|A_\sigma|)\big)\\&=&\mu(B_x)\log 3>0.
\end{eqnarray*}
It follows that $I_\mu(\mathbb{R}^q,\alpha_n)>I_\mu(\mathbb{R}^q,\beta)$, contradicting the optimality of $\alpha_n$.
\end{proof}

For every $n\geq (n_0+n_2)\phi_{k_0}$, there exists a unique $k\geq k_0$, such that
\begin{equation}\label{g2}
(n_0+n_2)\phi_k\leq n<(n_0+n_2)\phi_{k+1}.
\end{equation}
Recall that $n_3=(n_0+n_2)N$. By \cite[Lemma 2.1]{Zhu:20}, we have
\begin{equation}\label{ncomp}
(n_0+n_2)\phi_k\leq n<(n_0+n_2)N\phi_k=n_3\phi_k.
\end{equation}

From now on, we assume that $n,k$ satisfy (\ref{g2}). We fix an $\alpha_n\in C_n(\mu)$ and an arbitrary VP $\{P_a(\alpha_n)\}_{a\in\alpha_n}$. We write $D_{\omega,i}$ for $D_{\omega,i}^{(\alpha_n)}$. We define
\[
L_\sigma:={\rm card}(\alpha_n\cap(A_\sigma)_{\delta|A_\sigma|}),\;\sigma\in\Omega_k.
\]

Using the subsequent two lemmas, we establish a lower bound for $L_\sigma$.
\begin{lemma}\label{key1}
Let $\omega\in\Omega_k$ and $B\in\{D_{\omega,i}\}_{i=1}^3\cup\{E_\omega\}$. Then
\[
I_\mu(B,\alpha_n)\geq \mu(B)(\log|B|+\hat{e}_{L_\omega+L_1}(\nu_B)).
\]
\end{lemma}
\begin{proof}
We divide $B$ into two parts:
\[
B(1):=\{x\in B:d(x,\alpha_n)=d(x,\alpha_n\cap(A_\omega)_{\delta|A_\omega|)}\};\;B(2):=B\setminus B(1).
\]
By the definition, we have $B(2)\subset E_\omega$.  Let $\gamma_{E_\omega}$ be as defined in Remark \ref{rem4}. We define
$\gamma(\omega):=\big(\alpha_n\cap(A_\omega)_{\delta|A_\omega|}\big)\cup \gamma_{E_\omega}$.
Then ${\rm card}(\gamma(\omega))\leq L_\omega+L_1$ and clearly
$d(x,\alpha_n)\geq d(x,\gamma(\omega))$ for every $x\in B(1)$.
For $x\in B(2)$, we have
\[
d(x,\alpha_n))>\delta|A_\omega|=2\delta|E_\omega|>d(x,\gamma_{E_\omega})\geq d(x,\gamma(\omega)).
\]
Thus, for every $x\in B$, we have $d(x,\alpha_n)\geq d(x,\gamma(\omega))$. Thus, by Lemma \ref{l3},
\begin{eqnarray*}
I_\mu(B,\alpha_n)\geq I_\mu(B,\gamma(\omega))
\geq\mu(B)(\log|B|+\hat{e}_{L_\omega+L_1}(\nu_B)).
\end{eqnarray*}
This completes the proof of the lemma.
\end{proof}

Next, we give a lower bound for $\min\limits_{\sigma\in\Omega_k}L_\sigma$.
\begin{lemma}\label{lem7}
For every $\sigma\in\Omega_k$, we have $L_\sigma\geq n_1$.
\end{lemma}
\begin{proof}
Assume that $L_\sigma<n_1$ for some $\sigma\in\Omega_k$. We deduce a contradiction. By the assumption and Lemma \ref{lc}, we obtain
\[
\sum_{\tau\in\Omega_k\setminus\{\sigma\}}L_\tau>n-L_c-n_1\geq(n_2+n_0)\phi_k-n_0\phi_k-n_1>(\phi_k-1)n_2.
\]
Hence, there exists some $\omega\in\Omega_k\setminus\{\sigma\}$ such that $L_\omega>n_2$. We consider
\[
D_{\omega,2}=E_\omega\cup\bigg(\bigcup_{a\in\alpha_n\cap (A_\omega)_{\delta|A_\omega|}}(W(a|\alpha_n)\cap A_\omega^*)\bigg)\setminus E_\sigma.
\]
Note that it is possible that $E_\sigma\cap (A_\omega)_{\delta|A_\omega|}=\emptyset$. Let
\begin{eqnarray*}
&&\gamma_{L_\omega-L_0-n_1-L_1}(D_{\omega,2})\in C_{L_\omega-L_0-n_1-L_1}(\nu_{D_{\omega,2}}),\;\gamma_{n_1+L_1}(E_\sigma)\in C_{n_1+L_1}(\nu_{E_\sigma});
\\&&\beta:=\big(\alpha_n\setminus (A_\omega)_{\delta|A_\omega|}\big)\cup B_\omega\cup f_{D_{\omega,2}}(\gamma_{L_\omega-L_0-n_1-L_1}(D_{\omega,2}))\cup f_{E_\sigma}(\gamma_{n_1+L_1}(E_\sigma)).
\end{eqnarray*}
Then by applying Lemma \ref{compare2} to $F_{\omega,2}=D_{\omega,2}\cup E_\sigma$, we obtain
\begin{equation}\label{s4}
I_\mu(\mathbb{R}^q\setminus F_{\omega,2},\beta)\leq I_\mu(\mathbb{R}^q\setminus F_{\omega,2},\alpha_n).
\end{equation}
Next, we focus on the sets $D_{\omega,2}$ and $E_\sigma$. By the assumption, we have $L_\sigma\leq n_1-1$. Hence, by applying Lemmas \ref{l3} and \ref{key1} with $B=E_\sigma$, we deduce
\begin{eqnarray}
\Delta_{E_\sigma}:&=&I_\mu(E_\sigma,\alpha_n)-I_\mu(E_\sigma,\beta)\nonumber\\
&\geq&I_\mu(E_\sigma,\alpha_n)-I_\mu(E_\sigma,f_{E_\sigma}(\gamma_{n_1+L_1}(E_\sigma)))\nonumber\\
&\geq&\mu(E_\sigma)\big(\hat{e}_{n_1+L_1-1}(\nu_{E_\sigma})-\hat{e}_{n_1+L_1}(\nu_{E_\sigma})\big).
\end{eqnarray}
On the other hand, we apply Lemmas \ref{l3} and \ref{key1} with $B=D_{\omega,2}$, we have
\begin{eqnarray}
\Delta_{D_{\omega,2}}:&=&I_\mu(D_{\omega,2},\beta)-I_\mu(D_{\omega,2},\alpha_n)\nonumber\\
&\leq&I_\mu(D_{\omega,2},f_{D_{\omega,2}}(\gamma_{L_\omega-L_0-n_1-L_1}(D_{\omega,2})))-I_\mu(D_{\omega,2},\alpha_n)\nonumber\\
&\leq&\mu(D_{\omega,2})\big(\hat{e}_{L_\omega-L_0-n_1-L_1}(\nu_{D_{\omega,2}})-\hat{e}_{L_\omega+L_1}(\nu_{D_{\omega,2}})\big).
\end{eqnarray}
Note that $L_\omega>n_2$. By Lemma \ref{lem6} (d1) and Remark \ref{rem2}, we obtain that $\Delta_{E_\sigma}>\Delta_{D_{\omega,2}}$. This and (\ref{s4}), yield that $I_\mu(\mathbb{R}^q,\alpha_n)>I_\mu(\mathbb{R}^q,\beta)$, contradicting the optimality of $\alpha_n$.
\end{proof}

Our next lemma is very helpful for us to characterize the geometrical structure of an optimal VP.
\begin{lemma}\label{lem8}
For every $\sigma\in\Omega_k$, the following holds:
\[
\sup_{x\in A_\sigma\cap K_\mu}d(x,\alpha_n)\leq \delta|A_\sigma|.
\]
\end{lemma}
\begin{proof}
Assume that for some $\sigma\in\Omega_k$ and some $x_0\in A_\sigma\cap K_\mu$, we have, $d(x_0,\alpha_n)>\delta|A_\sigma|$. We will deduce a contradiction. By the assumption, for $x\in B(x_0,2^{-1}\delta|A_\sigma|)$, we have $d(x,\alpha_n)>2^{-1}\delta|A_\sigma|$. Let $G_{x_0}$ be as defined in Remark \ref{rem4} (c2). We consider
\[
D_{\sigma,1}=E_\sigma\cup\bigg(\bigcup_{b\in\alpha_n\cap(A_\sigma)_{\delta|A_\sigma|}}(W_b(\alpha_n)\cap A_\sigma^*)\bigg)\setminus B(x_0,2^{-1}\delta|A_\sigma|).
\]
Let $\gamma_{L_\sigma-L_0-L_2}(D_{\sigma,1})\in C_{L_\sigma-L_0-L_2}(\nu_{D_{\sigma,1}})$. We define
\begin{eqnarray*}
\gamma:=\big(\alpha_n\setminus(A_\sigma)_{\delta|A_\sigma|}\big)\cup G_{x_0}\cup B_\sigma\cup f_{D_{\sigma,1}}(\gamma_{L_\sigma-L_0-L_2}(D_{\sigma,1})).
\end{eqnarray*}
Then by applying Lemma \ref{compare2} to $F_{\sigma,1}=D_{\sigma,1}\cup B(x_0,2^{-1}\delta|A_\sigma|)$, we obtain
\begin{equation}\label{s1}
I_\mu(\mathbb{R}^q\setminus F_{\sigma,1},\gamma)\leq I_\mu(\mathbb{R}^q\setminus F_{\sigma,1},\alpha_n).
\end{equation}
For every $x\in B(x_0,2^{-1}\delta|A_\sigma|)$, we have $d(x,\gamma)\leq 4^{-1}\delta|A_\sigma|$. It follows that
\begin{eqnarray}\label{s2}
\Delta_{x_0}:&=&I_\mu(B(x_0,2^{-1}\delta|A_\sigma|),\alpha_n)-I_\mu(B(x_0,2^{-1}\delta|A_\sigma|),\gamma)\nonumber\\&\geq&\mu(B(x_0,2^{-1}\delta|A_\sigma|))\log 2\nonumber\\&\geq&C_1(2^{-1}\delta|A_\sigma|)^{s_0}\log 2\nonumber\\&\geq& C_1C_2^{-1} \delta^{s_0}\mu(A_\sigma)\log2.
\end{eqnarray}
By applying Lemmas \ref{l3} and \ref{key1} with $B=D_{\sigma,1}$, we have
\begin{eqnarray}\label{s3}
\Delta_{D_{\sigma,1}}:&=&I_\mu(D_{\sigma,1},\gamma)-I_\mu(D_{\sigma,1},\alpha_n)
\nonumber\\&\leq&I_\mu(D_{\sigma,1},f_{D_{\sigma,1}}(\gamma_{L_\sigma-L_0-L_2}(D_{\sigma,1})))-I_\mu(D_{\sigma,1},\alpha_n)\nonumber\\&\leq&
\mu(D_{\sigma,1})\big(\hat{e}_{L_\sigma-L_0-L_2})(\nu_{D_{\sigma,1}})-\hat{e}_{L_\sigma+L_1}(\nu_{D_{\sigma,1}})\big).
\end{eqnarray}
From Lemma \ref{lem7}, we know that $L_\sigma\geq n_1$. Thus, by Lemmas \ref{lem2}, \ref{lem10} and Remark \ref{rem2}, we obtain that $\Delta_{x_0}>\Delta_{D_{\sigma,1}}$. This and (\ref{s1}) imply that $I_\mu(\mathbb{R}^q,\gamma)<I_\mu(\mathbb{R}^q,\alpha_n)$, which contradicts the optimality of $\alpha_n$.
\end{proof}

\begin{remark}\label{rem5}
By Lemma \ref{lem8}, we obtain that, whenever $n\geq(n_0+n_2)\phi_k$, we have
$L_c={\rm card}(\alpha_n\setminus\bigcup_{\sigma\in\Omega_k}(A_\sigma)_{\delta|A_\sigma|})=0$. Therefore, we have
\[
\alpha_n\subset\bigcup_{\sigma\in\Omega_k}(A_\sigma)_{\delta|A_\sigma|}.
\]
\end{remark}

\begin{lemma}\label{lemma1}
Let $\emptyset\neq\beta\subset\mathbb{R}^q$ be a finite set and $l_\beta(\omega):={\rm card}(\beta\cap E_\omega)$ for $\omega\in\Omega_k$. Then the following estimate holds:
\[
I_\mu(D_{\omega,4},\beta)\geq\mu(D_{\omega,4})(\log|D_{\omega,4}|+\hat{e}_{l_\beta(\omega)+L_1}).
\]
\end{lemma}
\begin{proof}
Let $\gamma_{E_\omega}$ be as defined in Remark \ref{rem4} (c1). Since $D_{\omega,4}\subset E_\omega$, we have $d(x,\gamma_{E_\omega})\leq2^{-1}\delta|E_\omega|$ for every $x\in D_{\omega,4}$. We define
\[
\gamma(\omega):=(\beta\cap E_\omega)\cup\gamma_{E_\omega}.
 \]
Then ${\rm card}(\gamma(\omega))\leq l_\beta(\omega)+L_1$. Let $x\in D_{\omega,4}$. If $d(x,\beta)=d(x,\beta\cap E_\omega)$, then it is clear that $d(x,\gamma(\omega))\leq d(x,\beta)$. Otherwise, we have
\[
d(x,\beta)=d(x,\beta\setminus E_\omega)\geq\delta|E_\omega|>d(x,\gamma_{E_\omega})\geq d(x,\gamma).
\]
Thus, $I_\mu(D_{\omega,4},\beta)\geq I_\mu(D_{\omega,4},\gamma(\omega))$. The lemma follows by Lemma \ref{l3}.
\end{proof}

Now we are able to give an upper bound for $\max_{\sigma\in\Omega_k}L_\sigma$.
\begin{lemma}\label{lem4}
For every $\sigma\in\Omega_k$, we have $L_\sigma\leq n_4$.
\end{lemma}
\begin{proof}
Assume that, for some $\sigma\in\Omega_k$, we have $L_\sigma>n_4(>M_0n_3)$. Next, we deduce a contradiction. By the assumption and (\ref{ncomp}), we deduce
\[
{\rm card}(\alpha_n\setminus (A_\sigma)_{\delta|A_\sigma|})\leq n-n_4<n_3\phi_k-M_0n_3\leq(\phi_k-M_0)n_3.
\]
 By Lemma \ref{lem9}, we have ${\rm card}(\Omega_k\setminus\mathcal{A}_\sigma)\geq \phi_k-M_0$. Note that $E_\rho,\rho\in\Omega_k$, are pairwise disjoint. There exists some $\omega\in\Omega_k\setminus\mathcal{A}_\sigma$ such that
${\rm card}(\alpha_n\cap E_\omega)<n_3$. We consider
\[
D_{\sigma,3}=E_\sigma\cup\bigg(\bigcup_{a\in\alpha_n\cap(A_\sigma)_{\delta|A_\sigma|}}(W_a(\alpha_n)\cap A_\sigma^*)\bigg).
\]
Using Lemma \ref{lem8} and the triangle inequality, for every $x\in E_\omega$, we have
\begin{eqnarray*}
d(x,\alpha_n)\leq \frac{1}{2}|E_\omega|+\delta|A_\omega|<d(x,(A_\sigma)_{\delta|A_\sigma|}).
\end{eqnarray*}
It follows that $E_\omega\cap D_{\sigma,3}=\emptyset$. We define
\begin{eqnarray*}
&&\gamma_{n_3+L_1}(D_{\omega,4})\in C_{n_3+L_1}(\nu_{D_{\omega,4}}),\;\gamma_{L_\sigma-L_0-n_3-L_1}(D_{\sigma,3})\in C_{L_\sigma-L_0-n_3-L_1}(\nu_{D_{\sigma,3}});\\
&&\beta:=\big(\alpha_n\setminus(A_\sigma)_{\delta|A_\sigma|}\big)\cup B_\sigma\cup f_{D_{\sigma,3}}(\gamma_{L_\sigma-L_0-n_3-L_1}(D_{\sigma,3}))\cup f_{D_{\omega,4}}(\gamma_{n_3+L_1}(D_{\omega,4})).
\end{eqnarray*}
Then ${\rm card}(\beta)\leq n$. By applying Lemma \ref{compare2} to $F_{\sigma,3}=D_{\sigma,3}\cup D_{\omega,4}$, we obtain
 \begin{eqnarray}\label{s6}
I_\mu(\mathbb{R}^q\setminus F_{\sigma,3},\beta)\leq I_\mu(\mathbb{R}^q\setminus F_{\sigma,3},\alpha_n).
\end{eqnarray}
This allows us to focus on integrals over the sets $D_{\sigma,3}$ and $D_{\omega,4}$. Note that for every $x\in D_{\omega,4}$, we have $d(x,\beta)\leq d(x,\gamma_{n_3+L_1}(D_{\omega,4}))$. Applying Lemma \ref{l3} with $B=D_{\omega,4}$  and Lemma \ref{lemma1} , we have
\begin{eqnarray*}
\Delta_{D_{\omega,4}}:&=&I_\mu(D_{\omega,4},\alpha_n)-I_\mu(D_{\omega,4},\beta)\nonumber\\
&\geq&I_\mu(D_{\omega,4},\alpha_n)-I_\mu(D_{\omega,4},f_{D_{\omega,4}}(\gamma_{n_3+L_1}(D_{\omega,4})))\nonumber
\\&=&\mu(D_{\omega,4})\big(\hat{e}_{n_3-1+L_1}(\nu_{D_{\omega,4}})-\hat{e}_{n_3+L_1}(\nu_{D_{\omega,4}})\big).
\end{eqnarray*}
Similarly, for every $x\in D_{\sigma,3}$, we have $d(x,\beta)\leq d(x,\gamma_{L_\sigma-L_0-n_3-L_1}(D_{\sigma,3}))$. Thus, we apply Lemmas \ref{l3} and \ref{key1} with $B=D_{\sigma,3}$ and obtain
\begin{eqnarray*}
\Delta_{D_{\sigma,3}}:&=&I_\mu(D_{\sigma,3},\beta)-I_\mu(D_{\sigma,3},\alpha_n)\nonumber\\
&\leq&I_\mu(D_{\sigma,3},f_{D_{\sigma,3}}(\gamma_{L_\sigma-L_0-n_3-L_1}(D_{\sigma,3})))-I_\mu(D_{\sigma,3},\alpha_n)\nonumber
\\&=&\mu(D_{\sigma,3})\big(\hat{e}_{L_\sigma-L_0-n_3-L_1}(\nu_{D_{\sigma,3}})-\hat{e}_{L_\sigma+L_1}(\nu_{D_{\sigma,3}})\big).
\end{eqnarray*}
By the assumption, we have $L_\sigma>n_4$. Thus, from Lemmas \ref{lem2}, \ref{lem6} (d2) and Remark \ref{rem2}, we deduce that $\Delta_{D_{\omega,4}}>\Delta_{D_{\sigma,3}}$. Combining this with
(\ref{s6}), we obtain that $I_\mu(\mathbb{R}^q,\beta)<I_\mu(\mathbb{R}^q,\alpha_n)$, contradicting the optimality of $\alpha_n$.
\end{proof}

\section{Proof of Theorem \ref{mthm} }

Let $a\in\alpha_n$. By Remark \ref{rem5}, we have $a\in (A_\sigma)_{\delta|A_\sigma|}$ for some $\sigma\in\Omega_k$. Fix an arbitrary word $\tau_0\in\mathcal{A}_\sigma$. We define
\begin{eqnarray*}
&&\Gamma(\tau):=\alpha_n\cap (A_\tau)_{\delta|A_\tau|},\; \tau\in\mathcal{A}_\sigma;\\
&&G(a):=A_{\tau_0}\cup\bigcup_{\tau\in\mathcal{A}_\sigma}\bigcup_{b\in\Gamma(\tau)}(P_b(\alpha_n)\cap K_\mu);\\&&H(a):=\bigcup_{\tau\in\mathcal{A}_\sigma}\Gamma(\tau);\;T_a:={\rm card}(H(a)).
\end{eqnarray*}

Let $f_a$ be a similarity mapping of similarity ratio $|G(a)|$. We define
\[
\nu_{G(a)}=\mu(\cdot|G(a))\circ f_a=\mu\bigg(\cdot\bigg|\bigcup_{\tau\in\mathcal{A}_\sigma}\bigcup_{b\in\Gamma(\tau)}(P_b(\alpha_n)\cap K_\mu)\bigg)\circ f_a.
\]
\begin{lemma}\label{lem3}
Let $G(a)$ and $\nu_{G(a)}$ be as defined above. Then we have
\begin{enumerate}
\item[(i)]$P_a(\alpha_n)\cap K_\mu\subset G(a)$ and $n_1\leq T_a\leq M_0n_4=:n_5$;
\item[(ii)] there exists some constant $C$, such that
\[
\sup_{x\in\mathbb{R}^q}\nu_{G(a)}(B(x,\epsilon))\leq C\epsilon^{s_0}\;{\rm for\; every}\;\epsilon>0.
\]
\end{enumerate}
\end{lemma}
\begin{proof}
The first part of (i) is an easy consequence of the definition of $G(a)$. By Lemma \ref{lem9}, ${\rm card}(\mathcal{A}_\sigma)\leq M_0$. Further, for every $\tau\in\mathcal{A}_\sigma$, by Lemmas \ref{lem7} and \ref{lem4}, we have, $n_1\leq{\rm card}(\Gamma(\tau))\leq n_4$ for every $\tau\in\mathcal{A}_\sigma$. Hence, $n_1\leq T_a\leq n_5$.

Next, we show (ii). By the definitions of $G(a),\mathcal{A}_\rho$ and  $A_\rho,\rho\in\Omega_k$ and Lemma \ref{lem8}, we have
\begin{equation}\label{g1}
A_{\tau_0}\subset G(a)\subset\bigcup_{\tau\in\mathcal{A}_\sigma}\bigcup_{\rho\in\mathcal{A}_\tau}A_\rho\subset B\big(c_\sigma,(8\delta+\frac{5}{2})|A_{\tau_0}|\big)
\end{equation}
Thus, we have the following estimate:
\begin{equation}\label{z4}
|A_{\tau_0}|\leq|G(a)|\leq (5+16\delta)|A_{\tau_0}|.
\end{equation}
Let $\eta_3:= C_2(3+8\delta)^{s_0}$ and $\eta_4:=C_12^{-s_0}$. By (\ref{g1}), (\ref{z4}) and (\ref{AD}),
\begin{eqnarray}\label{compare1}
&&\mu(G(a))\leq C_2(3+8\delta)^{s_0}|A_{\tau_0}|^{s_0}\leq\eta_3|G(a)|^{s_0};\label{compare11}\\
&&\mu(G(a))\geq C_12^{-s_0}|A_{\tau_0}|^{s_0}\geq \eta_4(5+16\delta)^{-s_0}|G(a)|^{s_0}.\nonumber
\end{eqnarray}
Thus, from \cite[Lemma 2.5]{Zhu:20}, we obtain (ii).
\end{proof}

\emph{Proof of Theorem \ref{mthm}}

By (\ref{AD}), Lemmas \ref{lem11} and \ref{l1}, it is sufficient to consider $n\geq(n_0+n_2)\phi_{k_0}$.
Let $a\in\alpha_n$ and let $G(a),H(a),\nu_{G(a)}$ be as defined above. By Theorem 2.4 and Lemma 2.3 of \cite{GL:04} and the similarity of $f_a$, we know that $f_a^{-1}(H(a))\in C_{T_a}(\nu_{G(a)})$. From Lemma \ref{lem3} (i), we have that $n_1\leq T_a\leq n_5$. Because of Lemma \ref{lem3} (ii), we may apply Lemmas \ref{lem11}, \ref{l1} to the measure $\nu_{G(a)}$. We define
\[
\underline{d}:=\min_{2\leq h \leq n_5}\underline{d}_h,\;\overline{d}:=\max_{2\leq h \leq n_5}\overline{d}_h,\;d_3:=\min_{2\leq h \leq n_5}g_h.
\]
Thus, using the similarity of $f_a$ and Lemmas \ref{lem11} and \ref{l1}, we obtain
\[
\mu(G(a))\underline{d}\leq\mu(P_a(\alpha_n))=\mu(G(a))\nu_{G(a)}\big(P_{f_a^{-1}(a)}(f_a^{-1}(H(a)))\big)\leq \mu(G(a))\overline{d}.
\]
By Lemma \ref{lem5}, (\ref{compare11}) and (\ref{g2}), we have
\begin{eqnarray*}
&&\mu(G(a))\leq \eta_3|A_{\tau_0}|^{s_0}\leq\eta_3C_1^{-1}2^{s_0}\eta_2\phi_k^{-1}\leq n_3\eta_2\eta_3\eta_4^{-1}n^{-1};\\
&&\mu(G(a))\geq \mu(A_{\tau_0})\geq\eta_1\phi_k^{-1}\geq (n_0+n_2)\eta_1n^{-1}.
\end{eqnarray*}
It suffices to define $d_1:=\underline{d}\eta_1(n_0+n_2)$ and $d_2:=\overline{d}n_3\eta_2\eta_3\eta_4^{-1}$.

By Lemma \ref{l2}, we know that for every $b\in H(a)\setminus\{a\}$,
\[
d(f_a^{-1}(b),f_a^{-1}(a))\geq 3d_3|P_{f^{-1}(a)}(f_a^{-1}(H(a)))\cap K_{\nu_{G(a)}}|.
 \]
Thus, using the similarity of $f_a$, we obtain that
\[
\min_{b\in H(a)\setminus\{a\}}d(a,b)\geq 3d_3|P_a(\alpha_n)\cap K_\mu|.
\]
By Lemma \ref{lem8}, we know that $|P_a(\alpha_n)\cap K_\mu|\leq 2\delta|A_\sigma|$.
On the other hand, for every $b\in\alpha_n\setminus H(a)$, there exists some $\tau\in\Omega_k\setminus\mathcal{A}_\sigma$ such that $b\in (A_\tau)_{\delta|A_\tau|}$. Note that $(A_\tau)_{2\delta|A_\tau|}\cap(A_\sigma)_{2\delta|A_\sigma|}=\emptyset$ and $a\in (A_\sigma)_{\delta|A_\sigma|}$, we deduce
\[
d(b,a)\geq2\delta|A_\tau|=2\delta|A_\sigma|\geq |P_a(\alpha_n)\cap K_\mu|.
\]
Note that $3d_3<2^{-1}$. It follows that $d(b,a)\geq 3d_3|P_a(\alpha_n)\cap K_\mu|$ for every $b\in\alpha_n\setminus\{a\}$. Thus, the set $P_a(\alpha_n)$ contains a closed ball of radius $d_3|P_a(\alpha_n)\cap K_\mu|$ which is centered at $a$.



\begin{thebibliography}{16}
\bibitem{BW:82}J. A. Bucklew,  G.L. Wise, Multidimensional
asymptotic quantization with $r$th power distortion measures.
IEEE Trans. Inform. Theory 28 (1982), 239-247.

\bibitem{Ger:79}A. Gersho, Asymptotically optimal block quantization. IEEE Trans. Inform.
Theory 25 (1979), 373-380.

\bibitem{GL:00}S. Graf and H. Luschgy, Foundations of quantization  for
 probability distributions. Lecture Notes in Math., Vol. 1730,
 Springer-Verlag, 2000.
 \bibitem{GL:04}S. Graf and H. Luschgy, Quantization for probability measures with respect to the geometric
mean error. Math. Proc. Camb. Phil. Soc. 136 (2004) 687-717.

\bibitem{GL:08}S. Graf, H. Luschgy and  G. Pag\`{e}s, Distortion mismatch in the quantization
of probability measures. ESAIM Probability and Statistics 12 (2008), 127-153.

\bibitem{GL:12}S. Graf, H. Luschgy and  G. Pag\`{e}s, The local quantization behavior of absolutely continuous probabilities. Ann.  Probab. 40 (2012), 1795-1828.

\bibitem{GN:98}R. Gray, D. Neuhoff, Quantization. IEEE Trans. Inform. Theory 44 (1998), 2325-2383.

\bibitem{GPM:02}P. M. Gruber, Optimum quantization and its applications. Adv. Math. 186 (2004), 456-497.


\bibitem{KZ:15}M. Kesseb\"{o}hmer and S. Zhu, Some recent developments in quantization of fractal measures. Fractal Geometry and Stochastics V.Birkh\"{a}user, Cham., 105-120, 2015.
\bibitem{KZ:16}M. Kesseb\"{o}hmer and S. Zhu, On the quantization for self-affine measures on Bedford-McMullen carpets. Math. Z. 283 (2016), 39-58.

\bibitem{Kr:13}W. Kreitmeier, Asymptotic optimality of scalar Gersho quantizers. Constructive Approximation 38 (2013), 365-396.
\bibitem{Mattila:95}P. Mattila, Geometry of sets and measures in Euclidean spaces. Cambridge Univ. Press, 1995.

\bibitem{PG:97}G. Pag\`{e}s, A space quantization method for numerical integration. J. comput. Appl. Math. 89 (1997), 1-38.

\bibitem{PK:01}K. P\"{o}tzelberger, The quantization dimension of distributions.
Math. Proc. Camb. Phil. Soc. 131 (2001), 507-519.

\bibitem{PK:04}K. P\"{o}tzelberger, The quantization error of self-similar distributions.
Math. Proc. Camb. Phil. Soc. 137 (2004), 725-740.

\bibitem{Zhu:08}S. Zhu, Quantization dimension for condensation systems. Math. Z. 259 (2008), 33-43.


\bibitem{Zhu:13}S. Zhu, A characterization of the optimal sets for self-similar measures with respect to the geometric mean error. Acta Math. Hung. 138 (2013), 201-225.

\bibitem{Zhu:20}S. Zhu, Asymptotic uniformity of the quantization error for Ahlfors-David probability measures. Sci. China Math. 63 (2020), 1039-1056.







\end{thebibliography}
\end{document}